\documentclass[draft]{amsproc}
\usepackage{amsmath}
\usepackage{amsfonts}
\usepackage[dvips]{graphicx}
\usepackage{amssymb}
\usepackage{amsbsy}
\usepackage{amsthm}
\usepackage{graphics}
\usepackage{color}
\usepackage{textcomp}
\usepackage{ mathrsfs }
\usepackage{scalerel,stackengine}
\stackMath
\newcommand\reallywidehat[1]{%
\savestack{\tmpbox}{\stretchto{%
  \scaleto{%
    \scalerel*[\widthof{\ensuremath{#1}}]{\kern-.6pt\bigwedge\kern-.6pt}%
    {\rule[-\textheight/2]{1ex}{\textheight}}
  }{\textheight}%
}{0.5ex}}%
\stackon[1pt]{#1}{\tmpbox}%
}
\makeatletter
\@namedef{subjclassname@2010}{%
\textup{2010} Mathematics Subject Classification}
\makeatother

\numberwithin{equation}{section}

\newtheorem{theorem}{Theorem}[section]

\newtheorem{lemma}[theorem]{Lemma}

\newcommand{\ttt}{\mathscr{T}}
\newcommand{\ann}{\mathbb{A}}

\newcommand{\scd}{{T^{\prime*}}}
\newcommand{\cd}{T^{\prime}}

\newcommand{\hil}{\mathcal{H}}
\newcommand{\chil}{\mathscr{H}}

\newcommand{\ddd}{\mathcal{D}}
\newcommand{\natu}{\mathbb{N}}

\newcommand{\bou}{\mathbf{B}(\hil)}
\newcommand{\bouc}{\mathbf{B}(\chil)}

\newcommand{\sbou}{\mathbf{B}}
\newcommand{\boue}{\mathbf{B}(E)}
\newcommand{\spec}{\sigma(T)}

\newcommand{\rad}{r(T)}
\newcommand{\cal}{\mathbb{Z}}

\newcommand{\nul}{\mathcal{N}(}
\newcommand{\real}{\mathbb{R}}
\newcommand{\comp}{\mathbb{C}}

\newcommand{\disc}{\mathbb{D}}
\newcommand{\elu}{L^2(\mu)}
\newcommand{\pe}{P_E}
\newcommand{\mul}{\mathscr{M}_z}
\newcommand{\com}{C_{\varphi,w}}

\newcommand{\szift}{S_\lambda}

\newcommand{\jad}{\kappa_\chil}

\newcommand{\hp}{\hat{\varphi}}
\newcommand{\hpa}{\widehat{\varphi_A}}

\newcommand{\hps}{\hat{\psi}}
\newcommand{\mult}{M_{\hp}}
\newcommand{\multl}{M_{\hp}^{\lambda}}

\theoremstyle{definition}
\newtheorem{ex}[theorem]{Example}

\newcommand*{\Le}{\leqslant}

\newcommand{\ran}{\mathcal{R}(}
\newcommand{\la}{\langle}
\newcommand{\ra}{\rangle}
\newcommand{\gwon}{[E]_{T^*,\cd}}

\DeclareMathOperator{\gen}{Gen}
\DeclareMathOperator{\lin}{lin}
\DeclareMathOperator{\card}{card}
\DeclareMathOperator{\czil}{Chi}
\DeclareMathOperator{\pare}{par}
\DeclareMathOperator{\des}{Des}
\DeclareMathOperator{\ro}{root}
\DeclareMathOperator{\parr}{par}
\DeclareMathOperator{\intt}{int}

\newcommand{\cycle}{\mathscr{C}_{\phi}}

\newcommand{\set}[1]{\left\{#1\right\}}

\newcommand{\norm}[1]{\left\Vert#1\right\Vert}

\begin{document}

\title[Generalized  multipliers for left-invertible operators]{Generalized  multipliers for left-invertible operators \\and applications}
   \author[P. Pietrzycki]{Pawe{\l} Pietrzycki}
   \subjclass[2010]{Primary 47B20, 47B33; Secondary
47B37} \keywords{left-invertible operators, composition operator, weighted shift, weighted shift on directed three}
   \address{Wydzia{\l} Matematyki i Informatyki, Uniwersytet
Jagiello\'{n}ski, ul. {\L}ojasiewicza 6, PL-30348
Krak\'{o}w}
   \email{pawel.pietrzycki@im.uj.edu.pl}
   \begin{abstract} We introduce generalized multipliers for left-invertible operators which formal Laurent series $U_x(z) =\sum_{n=1}^\infty
(P_ET^{∗n}x)\frac{1}{z^n}+
\sum_{n=0}^\infty
(P_E{T^{\prime*}}^{∗n}x)z^n$ actually represent analytic functions on an annulus or a disc.
\end{abstract}
   \maketitle
   \section{Introduction}

In \cite{shi} S. Shimorin  obtain a weak analog
of the  Wold decomposition theorem, representing operator close to isometry in some sense as a direct sum
of a unitary operator and a shift operator acting in some reproducing kernel Hilbert space of vector-valued
holomorphic functions defined on a disc. The construction of
 the Shimorin's model for a left-invertible analytic operator $T\in \bou$ is as follows.
Let $E:=\nul T^*)$ and define a 
vector-valued  holomorphic  functions $U_x$ as
  \begin{equation*}
   U_x(z) =
\sum_{n=0}^\infty
(P_E{T^{\prime*}}^{∗n}x)z^n,\quad z\in \disc({r(\cd)}^{-1}),
 \end{equation*}
where $\cd$
is the Cauchy dual of $T$. Then
we equip the obtained space of analytic functions $\chil:=\set{U_x:x\in \hil}$
with the norm induced by $\hil$. The operator $ U:\hil\ni x\to U_x\in \chil$
becomes a unitary operator.
Moreover, Shimorin proved that $\chil$ is a  reproducing kernel Hilbert space and
 the operator $T$ is unitary equivalent to the operator $\mul$ of multiplication
by $z$ on $\chil$ and $\scd$ is unitary equivalent to the operator $\mathscr{L}$  given by the 
\begin{equation*}
     (\mathscr{L}f)(z)=\frac{f(z)-f(0)}{z}, \quad f\in\chil.
\end{equation*}
 Following \cite{shi}, the reproducing kernel for $\chil$ is an $\boue$-valued function of two variables $\jad:\Omega\times\Omega\rightarrow \boue$ that
   \begin{itemize}
       \item[(i)] for any $e\in E$ and $\lambda \in \Omega$
       \begin{equation*}
           \jad(\cdot,\lambda)e\in \chil
       \end{equation*}
       \item[(ii)]for any $e\in E$, $f\in \chil$ and $\lambda \in \Omega$
       \begin{equation*}
           \la f(\lambda),e\ra_E=\la f,\jad(\cdot,\lambda)e\ra_\chil
       \end{equation*}
   \end{itemize}


   The class of  weighted shifts on a directed tree was introduced in \cite{memo} and intensively studied
since then
\cite{chav,planeta,bdp}. The class is a
source of interesting examples (see e.g.,
\cite{dym,ja}). In \cite{chav} S. Chavan and S. Trivedi showed that a weighted shift $\szift$ on  a rooted directed tree with finite branching index is analytic therefore  can be modelled as a multiplication operator $\mul$ on a reproducing
kernel Hilbert space $\chil$ of $E$-valued holomorphic functions on  a disc centered
at the origin, where $E := \nul\szift^*)$. Moreover, they proved that the reproducing kernel associated with $\chil$
is multi-diagonal.

In \cite{bdp} P. Budzy{\'n}ski, P. Dymek and M. Ptak introduced the notion of multiplier algebra induced by a weighted shift. In \cite{dym} P. Dymek, A. P{\l}aneta and  M. Ptak extended this notion to the case of left-invertible analytic operators.

   \section{Preliminaries}
   In this paper, we use the following notation. The
fields of rational, real and complex
numbers are denoted by $\mathbb{Q}$,
$\mathbb{R}$ and $\mathbb{C}$, respectively. The
symbols $\mathbb{Z}$, $\mathbb{Z}_{+}$, $\mathbb{N}$
and $\mathbb{R}_+$ stand for the sets of integers,
positive integers, nonnegative integers,  and
nonnegative real numbers, respectively. Set 
$\disc(r)=\set{z\in \comp\colon |z|\Le r}$ and
$\ann(r^-,r^+)=\set{z\in \comp\colon r^-\Le |z|\Le r^+}$ for $r,r^-,r^+\in\real_+$.
The expression "a countable set" means a
finite set or a countably infinite set.

All Hilbert spaces considered in this paper are assumed to be complex. Let $T$ be a linear operator in a complex Hilbert
space $\mathcal{H}$. Denote by  $T^*$  the adjoint
of $T$.  We write $\bou$ 
 for the $C^*$-algebra of all bounded operators  and the
cone of all positive operators in $\mathcal{H}$, respectively.
The spectrum and spectral radius of $T\in\bou$ is denoted by $\spec$ and $\rad$
respectively.
Let $T \in\bou$.
We say that $T$ is \text{left-invertible} if there exists $S \in \bou$ such
that $ST = I$. The \textit{Cauchy dual operator} $\cd$ of a left-invertible operator $T\in \bou$ is defined by
\begin{equation*}
 \cd=T(T^*T)^{-1}   
\end{equation*}
The notion of the Cauchy dual operator has been introduced and studied by Shimorin
in the context of the wandering subspace problem for Bergman-type operators \cite{shi}. We call $T$ \textit{analytic} if $\hil_\infty:=\bigcap_{i=1}^\infty T^i=\set{0}$.

Let $X$ be a set and $\phi:X\to X$. If $n\in \natu$ then the $n$-th iterate of $\phi$ is given by $\phi^{(n)}=\phi\circ\phi\circ\dots\circ\phi$, $\phi$ composed with itself $n$-times. For $x\in X$ the set 
   \begin{equation*}
 [x]_\phi=\{y\in X: \text{there exist } i,j\in \natu \text{  such that  } \phi^{(i)}(x)=\phi^{(j)}(y) \}
   \end{equation*} 
   is
called the \textit{orbit of}  $f$ containing $x$. If $x\in X$ and $\phi^{(i)}(x)=x$ for some $i\in \mathbb{Z}_+$ then the \textit{cycle of} $\phi$ containing $x$ is the set
\begin{equation*}
  \cycle=\{\phi^{(i)}(x)\colon i\in \natu \}  
\end{equation*}
   Define the function $[\phi]:X\rightarrow \natu$ by
   \begin{itemize}
       \item[(i)] $[\phi](x)=0$ if  $x$ is in the cycle of $\phi$
       \item[(ii)] $[\phi](x^*)=0$, where $x^*$ is a fixed element of $X$ not containing a cycle,
       \item[(iii)]
     $[\phi](\phi(x))=[\phi](x)+1$ if $x$ is not in a cycle of $\phi$.
   \end{itemize}We set
    \begin{equation*}
    \gen_\phi{(m,n)}:=\{x\in X\colon m\Le[\phi](x)\Le n\}
\end{equation*}for $m,n\in\natu$.
   
Let $(X,\mathscr{A},\mu)$ be a $\mu$-finite measure space, $\phi: X \rightarrow X$ and $w:X \rightarrow \comp$ be
measurable transformations.
By a \textit{weighted composition 
operator} $\com$ in $\elu$ we mean a mapping
\begin{align}\label{kkomp} 
\ddd(\com)&=\{f\in \elu : w(f\circ\phi)\in \elu\},
 \\\notag
\com f&=w(f\circ\phi),\quad f \in\ddd(\com).
\end{align}

Let us recall some useful properties of composition operator we need
in this paper:

\begin{lemma}\label{podst}Let $X$ be a countable set, $\phi: X \rightarrow X$ and $w:X \rightarrow \comp$ be
measurable transformations. If $\com\in \sbou(\ell^2(X))$ then for any $x\in X$ and $n\in \natu$
\begin{itemize}
\item[(i)]$\com^*e_x=\overline{w(x)}e_{\varphi(x)}$
 \item[(ii)]$\com^{*n} e_x=\overline{w(x)w(\phi(x))\cdots w(\phi^{(n-1)}(x))}e_{\phi^{(n)}(x)}$,
       \item[(iii)]$\com^n e_x=\sum_{y\in\varphi^{-n}(x)}w(y)w(\phi(y))\cdots w(\phi^{(n-1)}(y))e_y$,
     \item[(iv)] $\com^*\com e_x=\Big(\sum_{y\in\varphi^{-1}(x)}|w(y)|^2\Big)e_x$.
   \end{itemize}
  
   \end{lemma}
    We now describe Cauchy dual of weighted composition operator
   \begin{lemma}\label{cdcom}Let $X$ be a countable set, $\phi: X \rightarrow X$ and $w:X \rightarrow \comp$ be
measurable transformations. If $\com\in \sbou(\ell^2(X))$ is left-invertible operator then the Cauchy dual $\com^\prime$ of $\com$ is also a weighted composition operator and is given by:
    \begin{equation*}
        \com^\prime e_x=\sum_{y\in\varphi^{-1}(x)}\frac{w(y)}{\Big(\sum_{z\in\varphi^{-1}(y)}|w(z)|^2\Big)}e_y.    \end{equation*}
   \end{lemma}
 Let $\ttt = (V; E)$   be a directed tree ($V$ and
$E$ are the sets of vertices and edges of $\ttt$,
respectively). For any vertex $u \in V$ we
put $\textup{Chi}(u) = \{v \in V : (u, v) \in E\}$. Denote by $\parr$ the partial  function from $V$ to $V$ which assigns to a vertex $u$  a unique $v\in  V$ such that $(v,u)\in E$.  A vertex $u \in V$ is called a root of $\ttt$ if $u$ has no parent. If $\ttt$ has a root, we denote it by
$\textrm{ root}$. Put $V^{\circ}= V
\setminus\{\textrm{root}\}$ if $\ttt$ has a root and
$V^{\circ}=V$ otherwise.   The Hilbert
space of square summable complex functions on $V$
equipped with the standard inner product is denoted by
$\ell^2(V )$. For $u \in V$, we define $e_u \in
\ell^2(V)$ to be the characteristic function of the
 set $\{u\}$. It turns out that the set $\set{e_v}_{v\in V}$ is an orthonormal basis of $\ell^2(V)$. We put $V_\prec:=\set{v\in V: \card(\czil(V))\geq2}$ and call the a member of this set a \textit{branching vertex} of $\ttt$

Given a system $ \lambda=\{\lambda_v \}_{v \in
V^{\circ}} $ of complex numbers, we define the
operator $S_\lambda$ in $\ell^2(V)$, which is called a
\textit{weighted shift} on $\ttt$ with weights
$\lambda$, as follows
   \begin{equation*}
\ddd(S_{\lambda}) =\{ f\in \ell^2(V ):
\varLambda_{\ttt} f \in \ell^2(V )\}\quad \textup{and}
\quad S_{\lambda} f = \varLambda_{\ttt} f \quad
\textup{for} \quad f \in \ddd(S_\lambda),
   \end{equation*}
where
   \begin{displaymath}
(\varLambda_{\ttt} f)(v) =\left\{\begin{array}{ll}
\lambda_v f(\textup{par}(v)) & \textrm{if  $v \in V^{\circ},$}\\
0 & \textrm{otherwise}.\end{array} \right.
 \end{displaymath}

\begin{lemma} \label{ker}  If $\szift$ is a densely defined weighted shift on a directed tree $\ttt$
with weights  $ \lambda=\{\lambda_v \}_{v \in
V^{\circ}} $, then
\begin{equation} \nul\szift^*) = \left\{ \begin{array}{ll}
\la e_{\ro} \ra\oplus\bigoplus_{u\in V_\prec}(\ell^2(\czil(u))\ominus\la \lambda^u\ra) & \textrm{if $\ttt$  has a root,}\\
\bigoplus_{u\in V_\prec}(\ell^2(\czil(u))\ominus\la \lambda^u\ra) & \textrm{
otherwise,}
\end{array} \right.
\end{equation}
where $\lambda^u\in\ell^2(\czil(u))$ is given by $\lambda^u:\ell^2(\czil(u))\ni v\to\lambda_v\in \comp$ 
\end{lemma}

A subgraph of
a directed tree $\ttt$ which itself is a directed tree will be called a subtree of $\ttt$. We refer the reader to \cite{memo} for more details on
weighted shifts on directed trees.

\section{Generalized  multipliers}
In the recent paper \cite{ja3} we introduced a new analytic model for left-invertible operators. Now, we  recall this model. Let $T\in \bou$  be a left-invertible operator and $E$ be a subspace of $\hil$ denote by $\gwon$ the direct sum of  the smallest $\cd$-invariant  subspace containing $E$ and the smallest $T^*$-invariant subspace containing $E$:
\begin{equation*}
\gwon:=\bigvee\{T^{*n}x\colon x\in E, n\in \natu\}\oplus\bigvee\{{\cd}^n x\colon x\in E, n\in \natu\},
\end{equation*}
where $\cd$ is the Cauchy dual of $T$. 

To avoid the repetition, we state the following assumption which will be used
frequently in this section.
   \begin{align} \tag{LI} \label{li}
\begin{minipage}{70ex} 
The operator  $T\in \bou$ is left-invertible   and $E$ is a subspace of $\hil$ such that $\gwon=\hil$.
\end{minipage}
    \end{align}

Suppose \eqref{li} holds. In this case we may construct a Hilbert $\chil$
 associated with $T$, of formal Laurent series with vector  coefficients. We proceed as follows. For
each $x \in\hil$, define a formal Laurent series $U_x$ with vector  coefficients  as  
   \begin{equation}\label{mod}
       U_x(z) =\sum_{n=1}^\infty
(P_ET^{∗n}x)\frac{1}{z^n}+
\sum_{n=0}^\infty
(P_E{T^{\prime*}}^{∗n}x)z^n.
 \end{equation}
   
 Let $\chil$ denote the vector space of formal Laurent series with vector  coefficients of the form $U_x$, $x \in \hil$.  
   Consider the map $U:\hil\rightarrow \chil$ defined by $Ux=U_x$. 
As shown in \cite{ja3}  $U$ is injective. In particular, we may equip the space
$\chil$ with the norm induced from $\hil$, so that $U$ is unitary.

By \cite{ja3} the operator $T$ is unitary equivalent to the operator $\mul:\chil\to\chil$ of multiplication
by $z$ on $\chil$ given by 
\begin{equation}
    (\mul f)(z)=zf(z),\quad f\in\chil 
\end{equation}
and operator $\scd$ is unitary equivalent to the operator $\mathscr{L}:\chil\to\chil$ given by
\begin{equation}\label{dualr}
    (\mathscr{L}f)(z)=\frac{f(z)-(P_{\nul \mul^*)}f)(z)}{z}, \quad f\in\chil.
\end{equation}
For left-invertible operator $T\in\bou$, among all subspaces satisfying condition \eqref{li} we will distinguish those subspaces $E$ which satisfy the following condition
\begin{equation}\label{prep}
   E\perp T^n E \qquad\text{and}\qquad E\perp {\cd}^n E, \qquad n\in \mathbb{Z}_+
\end{equation}
Observe that every $f\in\chil$ can be represented as follows
 \begin{equation*}
f=\sum_{n=-\infty}^\infty
\hat{f}(n)z^n,
 \end{equation*}
where
\begin{equation}\label{gg}
 \hat{f}(n) = \left\{ \begin{array}{ll}
\pe\scd^nU^*f & \textrm{if $n\in\natu$}\\
\pe T^{-n}U^*f & \textrm{if $n\in\cal\setminus\natu$.}
\end{array} \right.
 \end{equation}

\begin{lemma}\label{wsp}Let $\{f_n\}^\infty_{n=0}\subset\chil$ and $f\in\chil$ be such that $\lim_{n\to\infty} f_n=f$. Then
\begin{equation*}
\lim_{n\to\infty} \hat{f_n}(k)=\hat{f}(k)\quad \textrm{for } k\in \cal.
\end{equation*}
\end{lemma}
\begin{proof} It follows directly from \eqref{gg}.
\end{proof}
In \cite{bdp} P. Budzy{\'n}ski, P. Dymek and M. Ptak introduced the notion of multiplier algebra induced by a weighted shift. In \cite{dym} P. Dymek, A. P{\l}aneta,  M. Ptak extended this notion to the case of left-invertible analytic operators.

We introduce generalized multipliers for left-invertible  operators  which   formal  Laurent  series \eqref{mod} actually represent analytic functions on an annulus or a disc.
Define
the Cauchy-type multiplication $*:\boue^\cal\times E^\cal\to E^\cal$ given by
\begin{equation}\label{cauch}
    (\hat{\varphi}*\hat{ f})(n)=\sum^\infty _{k=-\infty}\hat{\varphi}(k)\hat{f}(n-k)\quad 
    \hat{\varphi}\in\boue^\cal,\: \hat{f}\in E^\cal.
\end{equation}
 We define
the operator  $\mult:\chil\supseteq\mathcal{D}(\mult)\to \chil$ by
\begin{align*}
  \mathcal{D}(\mult)=\{f\in\chil\colon &\textrm{there is }g\in\chil\textrm{  such that }\hat{\varphi}*\hat{ f}=\hat{g}\},\\
  &\mult f=g \textrm{ if }\hat{\varphi}* \hat{ f}=\hat{g}.
\end{align*}
 \begin{lemma}\label{uniq}
 Let $\hat{\varphi}:\cal\to\boue$. Then following assertions are satisfied:
 \begin{itemize}
     \item[(i)] for every $e\in E$ and $n\in\cal$, $(\hat{\varphi}*\widehat{(Ue)})(n)=\hat{\varphi}(n)e$
     \item[(ii)]
     \begin{align*}
         \widehat{\mult f}(n)&=\sum^\infty _{k=-\infty}\hat{\varphi}(n-k)\hat{f}(k)=\sum^\infty _{k=1}\hat{\varphi}(n+k)\pe T^{k}U^*f\\&+
         \sum^\infty _{k=0}\hat{\varphi}(n-k)\pe \scd^{k}U^*f, \quad f\in \mathcal{D}(\mult), n\in \cal
         \end{align*}
\item[(iii)]
\begin{align*}
(\mult f)(z)&=\hspace{-0.2cm}\sum^\infty_{n=-\infty}\hspace{-0.2cm}\Big(\sum^\infty _{k=1}\hat{\varphi}(n+k)\pe T^{k}U^*f+
         \sum^\infty _{k=0}\hat{\varphi}(n-k)\pe \scd^{k}U^*f\Big)z^n,\\& \: f\in \mathcal{D}(\mult).
\end{align*}
 \end{itemize}
 \end{lemma}
 \begin{proof}
 
 \begin{itemize}
 \item[(i)] Fix $e\in E$. By \eqref{prep} and \eqref{cauch}, we have
 \begin{align*}
      (\hat{\varphi}*\widehat{(Ue)})(n)=\sum^\infty_{k=1}\hat{\varphi}(n+k)\pe T^{k}e+
         \sum^\infty _{k=0}\hat{\varphi}(n-k)\pe \scd^{k}e=\hat{\varphi}(n)e, \quad n\in\cal.
 \end{align*}
 
 \end{itemize}
 \end{proof}
 We call $\hat{\varphi}$ a \textit{generalized
multiplier}  of $T$ and $\mult$ a \textit{generalized multiplication operator} if $\mult\in\bouc$. The set
of all generalized multipliers of the operator $T$
we denote   by
   $\mathcal{GM}(T)$.  One can easily verify that the set $\mathcal{GM}(T)$ is a linear subspace of $\boue^\cal$. Consider the map $V:\mathcal{GM}(T)\ni\hp\to\mult\in\bouc$. By Lemma \ref{uniq}, the kernel of $V$ is trivial. In particular, we may equip the  space $\mathcal{GM}(T)$ with the norm  $\norm{\cdot}:\mathcal{GM}(T)\to[0,\infty)$ induced from $\bouc$, so that $V$ is isometry:
 \begin{equation*}
     \norm{\hp}:=\norm{\mult},\quad \hp\in\mathcal{GM}(T).
 \end{equation*}
For operator $A\in\bou$ let $\hp_A:\cal\to\boue$ be a function defined by
 \begin{equation*}
     \hp_A(m) = \left\{ \begin{array}{ll}
\pe\scd^mA|_E & \textrm{if $m\in\natu$}\\
\pe T^{-m}A|_E & \textrm{if $m\in\cal\setminus\natu$}
\end{array} \right.
 \end{equation*}

\begin{theorem}\label{wla} Let $T$ be left-invertible. The following assertions are satisfied:
\begin{itemize}
\item[(i)]  For  $n\in\natu$ the sequence $\chi_{\{n\}}I_E$ is generalized multiplier and
$M_{\chi_{\{n\}}I_E}=\mul^n$. If $n\in\cal\setminus\natu$ then
 $\mathcal{D}(M_{\chi_{\{n\}}I_E})=\overline{\ran\mul)}$ and $M_{\chi_{\{n\}}I_E}=\mathcal{L}^{-n}$.
 
\item[(ii)] $\mul$ commutes with $\mult$, for $\hp\in\mathcal{GM}(T)$,
\item[(iii)]  $\hat{\varphi}*\hat{ \psi}\in \mathcal{GM}(T)$ for every $\hp,\hps\in\mathcal{GM}(T)$ and
\begin{equation*}
\mult M_{\hps}=M_{\hp*\hps},
\end{equation*}
\item[(iv)] The space $\mathcal{GM}(T)$  endowed with the Cauchy-type multiplication
\begin{equation}
    (\hat{\varphi}*\hat{ \psi})(n)=\sum^\infty _{k=-\infty}\hat{\varphi}(k)\hat{\psi}(n-k)\quad 
    \hat{\varphi},\hat{\psi}\in\boue^\cal,
\end{equation}
 is a Banach algebra.
\end{itemize}
\end{theorem}
\begin{proof}
\begin{itemize}
 \item[(i)]
 \end{itemize} 
Consider first the
case when $n\in \natu$. Fix $f\in \chil$ and  set $g=\mul^nf$ and $\hp=\chi_{\{n\}}I_E$. Then $\hat{\varphi}*\hat{ f}=\hat{g}$.
If $n\in\cal\setminus\natu$ and $f\in \overline{\ran\mul)}$ then by \eqref{dualr} we have $\mathscr{L}f=\frac{1}{z}f$. Define $g=\mathscr{L}^{-n}f$. As in the previous case we obtain $\hat{\varphi}*\hat{ f}=\hat{g}$. If $f\in \nul \mul^*)\setminus\set{0}$  then
\begin{equation*}
    \hat{f}( n) = \left\{ \begin{array}{ll}
0 & \textrm{if $\mathbb{Z}_+$}\\
P_ET^nU^*f & \textrm{if $n\in\cal\setminus\ \mathbb{Z}_+$}. 
\end{array} \right.
\end{equation*} Hence, $\hat{\varphi}*\hat{ f}(0)=0$ and there exist some $k\in \cal$ such that $\hat{\varphi}*\hat{ f}(k)\neq0$ which
contradicts \eqref{gg}.

\begin{itemize}
 \item[(ii)]
 \end{itemize}
   \begin{align*}
       (\mult\mul f)(z)&=\hspace{-0.2cm}\sum^\infty_{n=-\infty}\hspace{-0.2cm}\Big(\sum^\infty _{k=1}\hat{\varphi}(n+k)\pe T^{k}U^*\mul f+
         \sum^\infty _{k=0}\hat{\varphi}(n-k)\pe \scd^{k}U^*\mul f\Big)z^n
    \\&=\mul\sum^\infty_{n=-\infty}\hspace{-0.2cm}\Big(\sum^\infty _{k=1}\hat{\varphi}(n+k)\pe T^{k}U^* f+
         \sum^\infty _{k=0}\hat{\varphi}(n-k)\pe \scd^{k}U^* f\Big)z^n
         \\&=(\mul\mult f)(z),
   \end{align*}
 for   $f\in \chil$
 \begin{itemize}
   \item[(iii)]
\end{itemize}
\begin{align*}
  \reallywidehat{\mult M_{\hps}f}(n)&=\sum^\infty _{k=-\infty}\hat{\varphi}(k)\widehat{M_{\hps}f}(n-k)=\sum^\infty _{k=-\infty}\hat{\varphi}(k)\sum^\infty _{j=-\infty}\hps(j)\hat{f}(n-k-j)
  \\&=\sum^\infty_{k=-\infty}\hspace{-0.2cm}\hat{\varphi}(k)\sum^\infty _{l=-\infty}\hps(l-k)\hat{f}(n-l)
  =\sum^\infty _{k=-\infty}\sum^\infty _{l=-\infty}\hat{\varphi}(k)\hps(l-k)\hat{f}(n-l)
  \\&=\sum^\infty _{l=-\infty}\sum^\infty _{k=-\infty}\hat{\varphi}(k)\hps(l-k)\hat{f}(n-l)
  =\sum^\infty_{l=-\infty}\hp*\hps(l)\hat{f}(n-l)
  \\&=(\hp*\hps)*\hat{f}(n)\quad f\in\chil, n\in \cal.
\end{align*}

This implies that
$(\hp*\hps)*\hat{f}=\reallywidehat{\mult M_{\hps}f}$ for every $f\in\chil$. Hence $\mathcal{D}(M_{\hp*\hps})=\chil$ and $\mult M_{\hps}=M_{\hp*\hps}$. This in turn
implies that $\hp*\hps\in\mathcal{GM}(T)$, because $\mult, M_{\hps}\in\chil$.
\begin{itemize} 
\item[(iv)]
\end{itemize}
 Suppose that a sequence $\{\hp_n\}^\infty_{n=0}\subset\mathcal{GM}(T)$ is a Cauchy
sequence. Since the map $V:\mathcal{GM}(T)\ni\hp\to\mult\in\bouc$ is isometry, we see that the sequence $\{M_{\hp_n}\}^\infty_{n=0}\subset\bouc$ is also a Cauchy
sequence and there
exists
an
operator $\mathcal{A}\in\bouc$ such that $\lim_{n\to\infty} \hp_n=\mathcal{A}$. 
\end{proof}

\begin{theorem} Let $T\in\bou$ be left-invertible and $E\subset\hil$ be such that,
\begin{itemize}
    \item[(i)] $\gwon=\hil$ and $[E]_{\scd,T}=\hil$,
    \item[(ii)] 
    \begin{equation}\label{incl}
    T^{n}\scd^nE\subset E,\qquad n\in\natu,
\end{equation}
\item[(iii)] formal Laurent series \eqref{mod} converges absolutely in $E$ on $\Omega$ such $\intt \Omega\neq\emptyset$.
\item[(iv)] \eqref{prep} holds
\end{itemize}
If $A\in\bou$ commutes with $T$, then $\hp_A\in\mathcal{A}$ and $A=U^*M_{\hp_A}U$.
\end{theorem}
\begin{proof}
Let $\mathcal{A}=UAU^*$. All we need to prove is the following equality 
\begin{equation} \label{pop}
   (\hpa* \hat{f})(n)=\widehat{\mathcal{A}f}(n),\qquad f\in \chil, n\in \cal.
 \end{equation}
Fix $n\in \cal$. Consider first the
case when 
$f=UT^me$, $e\in E$,  $m\in\natu$. By \eqref{prep}
\begin{align*}
    \hp_A(n+k)\pe T^{k}U^*f&=\hp_A(n+k)\pe T^{k+m}e=\chi_{set{0}}(k+m)\hp_A(n)e,
\end{align*}
and
\begin{equation*}
    \hp_A(n-k)\pe \scd^{k}U^*f = \left\{ \begin{array}{ll}
\hp_A(n-k)\pe \scd^{k-m}e=0& \textrm{if $k>m$}\\
\hp_A(n-k)\pe T^{m-k}e=0 & \textrm{if $k<m$}\\
\hp_A(n-m)e & \textrm{if $k=m$}.\\
\end{array} \right.
 \end{equation*}

\begin{equation*}
   \hp_A(n-m)e = \left\{ \begin{array}{ll}
\pe\scd^{n-m}Ae=\pe\scd^nAT^me& \textrm{if $n>m$}\\
\pe T^{m-n}Ae=\pe\scd^nAT^me & \textrm{if $n<m$ and $n\geq 0$}\\
\pe T^{m-n}Ae=\pe T^{-n}AT^me & \textrm{if $n<m$ and $n<0$}

\end{array} \right.
 \end{equation*}
 This altogether implies that
\begin{align*}
   (\hpa*\hat{f})(n)&=\sum^\infty_{k=1}\hpa(n+k)\pe T^{k}U^*f+
   \sum^\infty_{k=0}\hpa(n-k)\pe \scd^{k}U^*f\\
   &=\chi_{ \natu}\pe\scd^nAT^me+\chi_{ \cal\setminus\natu}\pe\scd^nAT^me\\
   &=\widehat{Af}(n),
\end{align*}
where $f=UT^me$ for $e\in E$,  $m\in\natu$.

In turn, if
 $f=U\scd^me$ $m\in\cal_+$.  It is plain that
\begin{align*}
 \hp_A(n-k)\pe \scd^{k}U^*f &=\hp_A(n-k)\pe \scd^{k+m}e=0.
\end{align*}
It follows from \eqref{prep} and inclusion \eqref{incl} that
\begin{equation*}
 \hp_A(n+k)\pe T^{k}U^*f= \left\{ \begin{array}{ll}
\hp_A(n+k)\pe T^{k-m}(T^{m}\scd^me)=0& \textrm{if $k>m$}\\
\hp_A(n+k)\pe \scd^{m-k}(T^{m}\scd^me)=0 & \textrm{if $k<m$}\\
\hp_A(n+m)(T^{m}\scd^me)& \textrm{if $k=m$}.\\
\end{array} \right.
 \end{equation*}
Let $g_e=T^{m}\scd^me$ then
\begin{equation*}
   \hp_A(n+m)g_e = \left\{ \begin{array}{ll}
\pe\scd^{n+m}Ag_e=\pe\scd^nA\scd^me& \textrm{if $n+m\geq0$, $n\geq0$}\\
\pe\scd^{n+m}Ag_e=\pe T^{-n}A\scd^me & \textrm{if $n+m\geq0$, $n<0$}\\
\pe T^{-(m+n)}Ag_e=\pe T^{-n}A\scd^me & \textrm{if $n+m<0$}
\end{array} \right.
 \end{equation*}
 As a consequence, we have
 \begin{align*}
   (\hpa* \hat{f})(n)&=\sum^\infty_{k=1}\hpa(n+k)\pe T^{k}U^*f+
   \sum^\infty_{k=0}\hpa(n-k)\pe \scd^{k}U^*f\\
   &=\chi_{\natu}\pe\scd^nA\scd^me+\chi_{ \cal\setminus\natu}\pe\scd^nA\scd^me\\
   &=\widehat{Af}(n),
\end{align*}
where $f=U\scd^me$, for $e\in E$,  $m\in\natu$.
 We extend the previous equality  by linearity
 to the following space 
 \begin{align*}
   \lin\{UT^{n}x\colon x\in E, n\in \natu\}&\oplus\lin\{U{\scd}^n x\colon x\in E, n\in \natu\}.
\end{align*}
 An application of Lemma \ref{wsp} gives \eqref{pop} which completes the proof.
\end{proof}
It is interesting to observe that the class of left-invertible and analytic operators and the class of weighted shift on leafless directed trees satisfy the assumptions of the previous theorem.
 \begin{ex} Let $T\in\bou$ be left-invertible and analytic and $E:=\nul T^*)$. By \cite[Proposition 2.7]{shi},  $\hil_\infty^\perp=[E]_{\cd}$. Since $T$ is analytic  $\hil_\infty=\{0\}$, we see that $[E]_{\scd,T}\supset[E]_{T}=\hil$ and $\gwon\supset[E]_{\cd}=\hil$, which yields $\gwon=\hil$ and $[E]_{\scd,T}=\hil$. 
 Using equality $E=\nul T^*)$ one can show that \eqref{prep} holds and 
$T^{n}\scd^nE=\set{0}\subset E$
 \end{ex}
 \begin{ex} Let $\ttt$ be a rootless and leafless directed tree and $\lambda=\set{\lambda_v}_{v\in V}$ be a system of weights. 
Let $\szift$ be the weighted shift on $\ttt$ and $E:=\la \omega\ra\oplus\nul\szift^*)$. Assume that $\szift\in \bou$ and formal Laurent series \eqref{mod} converges absolutely in $E$ on $\Omega$ such $\intt \Omega\neq\emptyset$. By \cite[Lemma 4.2]{ja3} and \cite[Lemma 4.3.1]{9} we have $\gwon=\hil$ and $[E]_{\scd,T}=\hil$, where $T=\szift$. It is a matter of routine to verify that $T^{n}\scd^nE\subset\set{\omega}\subset E$.
 \end{ex}

\section{Weighted shifts on directed trees}

In \cite{bdp} P. Budzy{\'n}ski,  P. Dymek,  M. Ptak. introduced  a notion of a multiplier algebra induced by a weighted shift, which
is defined via related multiplication operators. Assume that
\begin{align}\label{gwi}
\begin{minipage}{70ex} $\ttt=(V,E)$ is a countably infinite rooted and leafless directed tree,
and $\lambda = \{\lambda_v\}_{v\in V^\circ} \subset(0,\infty)$
\end{minipage}
\end{align}
For $u\in V$ and $v\in\des(u)$ we set
\begin{equation*}
    \lambda_{u|v} = \left\{ \begin{array}{ll}
1 & \textrm{if $u=v$}\\
\prod_{n=0}^{k-1}\lambda_{\pare^n(v)}& \textrm{if $\pare^k(v)=u$.}
\end{array} \right.
\end{equation*}
Let $\hp:\natu\to\comp$.  Define the mapping $\Gamma^{\lambda}_{\hp}:\comp^V\to \comp^V$  by
\begin{equation*}
    (\Gamma^{\lambda}_{\hp})(v)=\sum_{k=0}^{|v|}\lambda_{\pare^k(v)|v}\hp(k)f(\pare^k(v)),\quad v\in V.
\end{equation*}
The \textit{multiplication operator}
$\multl:\ell^2(V)\supseteq\mathcal{D}(\multl)\to \ell^2(V)$  is given by
\begin{align*}
  \mathcal{D}(\multl)&=\{f\in\ell^2(V)\colon\Gamma^{\lambda}_{\hp} f \in \ell^2(V) \},\\
  \multl f&=\Gamma^{\lambda}_{\hp} f, \quad f\in \mathcal{D}(\multl)
\end{align*}
We can write the above definition of $\multl$ in the following form
\begin{equation}
    \multl f=\sum_{n=0}^\infty\hp(k)\szift^kf
\end{equation}

\begin{theorem} The following equality holds
\begin{equation*}
 (M_{\hps} f)(z)=\sum_{k=1}^\infty\hps(-k)\szift^{'*k}f+\sum_{k=0}^\infty\hps(k)\szift^kf\qquad \textup{for}\: f\in \ran\szift) 
 \end{equation*}
\end{theorem}

\begin{proof}
First, we note that
\begin{equation*}
  \sum_{n=-\infty}^\infty\hat{f}(n-k)z^n   = \left\{ \begin{array}{ll}
\mul^kf & \textrm{if $k\in\natu$ and $f\in\chil$}\\
\mathscr{L}^{-k}f & \textrm{if $k\in\cal\setminus\natu$ and $f\in \overline{\ran \mul)}$.}
\end{array} \right.
\end{equation*}
Let $\hp$ be a sequence with finite support. By Theorem \ref{wla} the sequence $\hp$ induces bounded operator $\mult$ on subspace $\overline{\ran\mul)}$. Changing order of summation we obtain
\begin{align*}
    (M_{\hps} f)(z)&=\sum_{n=-\infty}^\infty\sum_{k=-\infty}^\infty\hps(k)\hat{f}(n-k)z^n=\sum_{k=-\infty}^\infty\sum_{n=-\infty}^\infty\hps(k)\hat{f}(n-k)z^n\\&=\sum_{k=-\infty}^\infty\hps(k)\sum_{n=-\infty}^\infty\hat{f}(n-k)z^n\\&=\sum_{k=-\infty}^\infty\hps(k)(\chi_{\natu}(k)\mul^kf+\chi_{\cal\setminus\natu}(k)\mathscr{L}^{-k}f)\\&=\sum_{k=0}^\infty\hps(k)\mul^kf+\sum_{k=1}^\infty\hps(-k)\mathscr{L}^{k}f.
\end{align*}
Let, for $n\in\natu$  , denote by $\widehat{p_n}:\cal\to\comp$ the coefficients of the $n$-th Fejer kernel, i.e.,
\begin{equation*}
\widehat{p_n}(m)= \left\{ \begin{array}{ll}
1-\frac{m}{n+1} & \textrm{if $|m|\leq n$}\\
0 & \textrm{if $m>n$}
\end{array} \right.
\end{equation*}
As in the proof of \cite[Proposition 15]{dym} one can show that $M_{\hat{p_n}\hp}\stackrel{SOT}{\to} M_{\hp}$ in $\overline{\ran\mul)}$. 

\end{proof} 


 \bibliographystyle{amsalpha}
   
\end{document}